\documentclass[12pt,twoside]{amsart}
\usepackage{amsmath, amsthm, amscd, amsfonts, amssymb, graphicx}
\usepackage{enumerate}
\usepackage[colorlinks=true,
linkcolor=blue,
urlcolor=cyan,
citecolor=red]{hyperref}
\usepackage{mathrsfs}
\newtheorem{theorem}{Theorem}[section]
\newtheorem{lemma}{Lemma}[section]
\newtheorem{remark}{Remark}[section]

\newtheorem{corollary}{Corollary}[section]
\newtheorem{example}{Example}[section]
\newtheorem{proposition}{Proposition}[section]
\numberwithin{equation}{section}

\begin{document}
	
\title{New orders among Hilbert space operators}
\author{Mohammad Sababheh and Hamid Reza Moradi}
\subjclass[2010]{47A08,47A12,47A30,47A60}
\keywords{Loewner partial ordering, block matrix, numerical radius, hyponormal operator, $(\alpha,\beta)$-normal operator, singular values}

%------------------------------------------------------------------------------------%
\pagestyle{myheadings}
\markboth{\centerline {}}
{\centerline {}}
\bigskip
\bigskip

%------------------------------------------------------------------------------------%
%------------------------------------------------------------------------------------% 

\begin{abstract}
This article introduces several new relations among related Hilbert space operators. In particular, we prove some L\"{o}ewner partial orderings among $T, |T|, \mathcal{R}T, \mathcal{I}T, |T|+|T^*|$ and many other related forms, as a new discussion in this field; where $\mathcal{R}T$ and $\mathcal{I}T$ are the real and imaginary parts of the operator $T$. Our approach will be based on proving the positivity of some new matrix operators, where several new forms for positive matrix operators will be presented as a key tool in obtaining the other ordering results. As an application, we present some results treating numerical radius inequalities in a way that extends some known results in this direction, in addition to some results about the singular values.
\end{abstract}
\maketitle

\section{Introduction}
Let $\mathcal{H}$ be a complex Hilbert space, endowed with the inner product $\left<\cdot,\cdot\right>$, and let $\mathcal{B}(\mathcal{H})$ denote the $C^*-$algebra of all bounded linear operators on $\mathcal{H}$.   An operator $T\in\mathcal{B}(\mathcal{H})$ is said to be positive semi-definite if $\left<Ax,x\right>\geq 0$ for all $x\in\mathcal{H}$. Such operator is then denoted by $A\geq O$. If $A\geq O$ is invertible, we write $A>O.$ Unlike real numbers, the algebra $\mathcal{B}(\mathcal{H})$ is not totally ordered. That is, if $A\in\mathcal{B}(\mathcal{H})$, then it is not necessarily that $A\geq O$ or $-A\geq O.$ A possible ordering among elements of $\mathcal{B}(\mathcal{H})$ is the so called L\"owner partial ordering; where we say that $A\geq B$ for two self-adjoint operators $A,B\in\mathcal{B}(\mathcal{H}),$ if $A-B\geq O.$

Among the most well established operator inequalities is the well known arithmetic-geometric mean inequality which has more than one form. If $\|\cdot\|$ denotes the usual operator norm on $\mathcal{B}(\mathcal{H})$, that is $\|T\|=\sup_{\|x\|=1}\|Tx\|$, then for any two operators $S,T$ one has the inequality \cite{bhakitt}
\begin{equation}\label{eq_amgm_norm}
\|ST^*\|\leq\frac{1}{2}\|\;|S|^2+|T|^2\|
\end{equation}
where $T^*$ denotes the conjugate of $T$ and $|T|$ is the unique positive root of $T^*T.$ This inequality is referred to as an arithmetic-geometric mean inequality since it extends the scalar inequality $ab\leq\frac{a^2+b^2}{2},$ for the real numbers $a,b$. Although \eqref{eq_amgm_norm} is true, its L\"owner version is not. That is, we cannot have
\begin{equation}\label{eq_amgm_order_wrong}
ST^*\leq\frac{1}{2}(|S|^2+|T|^2)
\end{equation}
in general. A simple reasoning here is that $ST^*$ is not necessarily self-adjoint, so the ordering in \eqref{eq_amgm_order_wrong} does not make sense, to begin with.

If $T\in\mathcal{B}(\mathcal{H})$, the real and imaginary parts of $T$ are defined by $\mathcal{R}T=\frac{T+T^*}{2}$ and $\mathcal{I}T=\frac{T-T^*}{2i}.$ A possible extension of \eqref{eq_amgm_norm} would be to compare $2\mathcal{R}(ST)$ with $|S|^2+|T|^2$. We notice here that the triangle inequality immediately implies
\begin{align*}
2\|\mathcal{R}T\|\leq \|T\|+\|T^*\|.
\end{align*}

At this point, it might be asked about the validity of the stronger version
\begin{equation}\label{eq_R_question}
2|\mathcal{R}T|\leq |T|+|T^*|, T\in\mathcal{B}(\mathcal{H}).
\end{equation}
Unfortunately, this inequality is wrong in general, as one can easily check the example
$$T=\left[\begin{array}{ccc}0&1&0\\0&0&1\\0&0&0\end{array}\right].$$ 
To see this, we note that
$$\lambda_2(|2\mathcal{R}T|)=\lambda_2^{\frac{1}{2}}\left(\left[\begin{array}{ccc}1&0&1\\0&2&0\\0&0&1\end{array}\right]\right)=\sqrt{2}\not\leq 1=\lambda_2\left(\left[\begin{array}{ccc}1&0&0\\0&2&0\\0&0&1\end{array}\right]\right)=\lambda_2(|T|+|T^*|).$$
In fact, even the simpler inequality $\mathcal{R}T\leq |T|$ is not true in general, as one can easily check with the example $$T=\left[\begin{array}{cc}0&1\\0&0\end{array}\right]$$ since we have $$\left| T \right|-\mathcal RT=\left[ \begin{matrix}
   0 & -\frac{1}{2}  \\
   -\frac{1}{2} & 1  \\
\end{matrix} \right]\ngeq O.$$
The latter matrix is not positive since its eigenvalues are $\frac{-\sqrt{2}+1}{2}$ and $\frac{\sqrt{2}+1}{2}$.

Thus, it  is valid to search possible 
orderings (without the norm) between 
\begin{itemize}
\item $\mathcal{R}T$ and $|T|.$
\item $2\mathcal{R}T$ and $|T|+|T^*|.$
\item $2\mathcal{R}(ST)$ and $|S|^2+|T|^2.$
\end{itemize}
The sole goal of this paper is to discuss this problem more broadly, leading to several new relations in a more generalized form. However, we will deal with particular forms of matrix operators that imply the desired links. Our discussion will also lead to exciting relations among the singular values and numerical radius inequalities. Here we recall that the numerical radius of an operator $T$ is defined by
$$\omega(T)=\sup\{|\left<Tx,x\right>|: ||x||=1\}.$$

Our method that we use to
prove the desired results is based  mainly on block techniques. For this, we need some results from the literature, as follows.

\begin{lemma}\label{1}
\cite[Chapter 1]{2} The following statements are mutually equivalent, for $A,B,C\in\mathcal{B}(\mathcal{H})$:
\begin{itemize}
\item[(i)] $\left[ \begin{matrix}
   A & C  \\
   {{C}^{*}} & B  \\
\end{matrix} \right]\ge O$.
\item[(ii)] $\left[ \begin{matrix}
   B & {{C}^{*}}  \\
   C & A  \\
\end{matrix} \right]\ge O$.
\end{itemize}
\end{lemma}

\begin{lemma}\label{4}
\cite[Theorem 3.4]{1} Let $C$ be self-adjoint and $\left[ \begin{matrix}
   A & C  \\
   {{C}} & B  \\
\end{matrix} \right]\ge O$. Then $A,B> O$ and
\[\pm C\le A\sharp B\]
where the geometric mean $A\sharp B$ for $A,B> O$ is defined as
$A\sharp B={{A}^{\frac{1}{2}}}{{\left( {{A}^{-\frac{1}{2}}}B{{A}^{-\frac{1}{2}}} \right)}^{\frac{1}{2}}}{{A}^{\frac{1}{2}}}$.
\end{lemma}

\begin{lemma}\label{20}
\cite[Lemma 1]{3} Let $A,B,C\in \mathcal B\left( \mathcal H \right)$, where $A,B\geq O$. Then 
	\[\left[ \begin{matrix}
   A & {{C}^{*}}  \\
   C & B  \\
\end{matrix} \right]\ge O\Leftrightarrow {{\left| \left\langle Cx,y \right\rangle  \right|}^{2}}\le \left\langle  A x,x \right\rangle \left\langle  B y,y \right\rangle ;x,y\in \mathcal H.\]
\end{lemma}
\begin{remark}
It follows from Lemma \ref{20} that if $\left[ \begin{matrix}
   A & {{C}^{*}}  \\
   C & B  \\
\end{matrix} \right]\ge O$, then for any $x\in\mathcal{H}$,
\begin{align*}
|\left<Cx,x\right>|&\leq \sqrt{\left<Ax,x\right>\left<Bx,x\right>}\\
&\leq \frac{\left<Ax,x\right>+\left<Bx,x\right>}{2}\\
&=\left<\frac{A+B}{2}x,x\right>.
\end{align*}
Consequently, for any $x\in\mathcal{H}$, $
\mathcal{R}\left<Cx,x\right>\leq \left<\frac{A+B}{2}x,x\right>.$ But since $\mathcal{R}\left<Cx,x\right>=\left<\mathcal{R}Cx,x\right>,$ it follows that
$\mathcal{R}C\leq \frac{A+B}{2}.$ Further, since $|\left<Cx,x\right>|=|\left<-Cx,x\right>|,$ it follows that $-\mathcal{R}C\leq \frac{A+B}{2}.$ Thus, we conclude the following implication
\begin{equation}\label{eq_re_neg}
\left[ \begin{matrix}
   A & {{C}^{*}}  \\
   C & B  \\
\end{matrix} \right]\ge O\Rightarrow \pm\mathcal{R}C\leq \frac{A+B}{2}.
\end{equation}
\end{remark}
\begin{lemma}\label{010}\cite[Lemma 1.4]{zhan_book}
Let $A,B\geq O$ be invertible operators. Then $\left[ \begin{matrix}
   A & C  \\
   {{C}^{*}} & B  \\
\end{matrix} \right]\geq O$ if and only if $C{{B}^{-1}}{{C}^{*}}\le A$.
\end{lemma}

\begin{lemma}\label{16}
 \cite[Lemma 3.1]{1} Let $A_1,A_2, B_1,B_2, C\in \mathcal B\left( \mathcal H \right)$, where $A_1,A_2, B_1,B_2\ge O$. Then
\[\left[ \begin{matrix}
   {{A}_{i}} & C  \\
   {{C}^{*}} & {{B}_{i}}  \\
\end{matrix} \right]\ge O\left( i=1,2 \right)\Rightarrow \left[ \begin{matrix}
   {{A}_{1}}\sharp{{A}_{2}} & C  \\
   {{C}^{*}} & {{B}_{1}}\sharp{{B}_{2}}  \\
\end{matrix} \right]\ge O.\]
\end{lemma}

The following lemma treats the singular values of a matrix. By $\mathcal{M}_n$, we refer to the algebra of complex $n\times n$ matrices. Although this lemma is stated originally for compact operators, we limit it here to matrices. In the sequel, if $T\in\mathcal{M}_n$, $s_j(T)$ will denote the $j^{\text{th}}$ largest eigenvalue of $|T|.$
\begin{lemma}\label{0}
\cite[Theorem 2.1]{14} Let $A,B,C\in \mathcal{M}_n$ be such that 	$\left[ \begin{matrix}
   A & {{C}^{*}}  \\
   C & B  \\
\end{matrix} \right]\ge O$. Then ${{s}_{j}}\left( C \right)\le {{s}_{j}}\left( A\oplus B \right)$ for $j=1,2,\ldots ,n$. Here we denote the block matrix $\left[ \begin{matrix}
   A & O  \\
   O & B  \\
\end{matrix} \right]$ by $A\oplus B$.
\end{lemma}

\begin{lemma}\cite{13}\label{tao}
Let $A,B,C\in \mathcal{M}_n$ be such that 	$\left[ \begin{matrix}
   A & {{C}^{*}}  \\
   C & B  \\
\end{matrix} \right]\ge O$. Then 
\[2{{s}_{j}}\left( C \right)\le {{s}_{j}}\left( \left[ \begin{matrix}
   A & {{C}^{*}}  \\
   C & B  \\
\end{matrix} \right] \right)\]
for $j=1,2,\ldots ,n$. 
\end{lemma}
Related to the singular values of matrices, we recall that a unitarily invariant norm $\|\cdot\|_u$ on $\mathcal{M}_n$ is a matrix norm that satisfies $\|UAV\|_u=\|A\|_u$ for all $A\in\mathcal{M}_n$ and all unitary matrices $U,V.$

In treating the inequality $\mathcal{R}T\leq |T|$, we present a class that satisfies this relation. Here, we recall that an operator $T\in\mathcal{B}(\mathcal{H})$ is said to be semi-hyponormal if $|T^*|\leq |T|$; see \cite{6}. For example, normal operators are semi-hyponormal.

We will be also interested in the so called $(\alpha,\beta)$-normal operators. For real numbers $\alpha $ and $\beta $ with $0< \alpha \le 1\le \beta $, an operator $T\in \mathcal B\left( \mathcal H \right)$  is called $\left( \alpha ,\beta  \right)$-normal  \cite{8}, if
\begin{equation}\label{21}
{{\alpha }^{2}}{{\left| T \right|}^{2}}\le {{\left| {{T}^{*}} \right|}^{2}}\le {{\beta }^{2}}{{\left| T \right|}^{2}}.
\end{equation}
According to this definition, if  $T$ is $\left( \alpha ,\beta  \right)$-normal operator, then $T$ and ${{T}^{*}}$ majorize each other.

\section{Operator Inequalities}

In this section, we present our main results. This will be done in two subsections, where we study inequalities involving one operator in the first section, then we discuss two-operator inequalities.

\subsection{One-operator inequalities}
From \cite[Corollary 1.3.8 ]{2}, we know that if $T$ is a normal operator, then 
\begin{equation}\label{7}
\left[ \begin{matrix}
   \left| T \right| & {{T}^{*}}  \\
   T & \left| T \right|  \\
\end{matrix} \right]\ge O.
\end{equation}
This inequality can be extended to semi-hyponormal operators as follows.
	
\begin{proposition}\label{03}
Let $T\in \mathcal B\left( \mathcal H \right)$ be a semi-hyponormal operator. Then \eqref{7} holds.
\end{proposition}
\begin{proof}
By the mixed Cauchy-Schwarz inequality \cite{7}, we have for any $x,y\in \mathcal H$
\[{{\left| \left\langle Tx,y \right\rangle  \right|}^{2}}\le \left\langle \left| T \right|x,x \right\rangle \left\langle \left| {{T}^{*}} \right|y,y \right\rangle.\]
Since  $T$ is a semi-hyponormal, then
\[{{\left| \left\langle Tx,y \right\rangle  \right|}^{2}}\le \left\langle \left| T \right|x,x \right\rangle \left\langle \left| T \right|y,y \right\rangle\]
which is equivalent to \eqref{7}, thanks to Lemma \ref{20}.
\end{proof}

Now we are ready to present the following possible relation between $\mathcal{R}T$ and $|T|,$ for semi-hyponormal operators. 
\begin{theorem}\label{3}
Let $T\in \mathcal B\left( \mathcal H \right)$ be a semi-hyponormal operator. Then
	\[\pm\mathcal RT\le \left| T \right|.\]
\end{theorem}
\begin{proof}
{\it {First proof:}} By Proposition \ref{03} and Lemma \ref{1}, we have
\begin{equation}\label{29}
\left[ \begin{matrix}
   \left| T \right| & {{T}^{*}}  \\
   T & \left| T \right|  \\
\end{matrix} \right]\ge O\text{ and }\left[ \begin{matrix}
   \left| T \right| & T  \\
   {{T}^{*}} & \left| T \right|  \\
\end{matrix} \right]\ge O.
\end{equation}
Therefore, by adding the latter two matrix operators,
\begin{equation}\label{30}
\left[ \begin{matrix}
   \left| T \right| & \mathcal RT  \\
   \mathcal RT & \left| T \right|  \\
\end{matrix} \right]\ge O.
\end{equation}
This implies,
	\[\pm\mathcal RT\le \left| T \right|\]
thanks to Lemma \ref{4}, which completes the proof.

{\it{Second proof:}} Using \cite[Theorem 3]{3}, we have
$$|\left<Tx,x\right>|\leq \left<|T|x,x\right>,\;\forall x\in\mathcal{H}.$$ 
Thus we have
$$\left<\pm\mathcal{R}T x,x\right>=\pm\mathcal{R}\left<Tx,x\right>\leq |\left<Tx,x\right>|\leq \left<|T|x,x\right>,\;\forall x\in\mathcal{H}.$$
Hence, $\pm\mathcal{R}T\leq |T|.$
\end{proof}

We remind the reader that for a general $T$, the inequality $\pm\mathcal{R}T\leq |T|$ is not true, explaining the significance of Theorem \ref{3} for a particular class of operators.

Semi-hyponormal operators satisfy further interesting inequalities. In the next result, we present a singular-value inequality satisfied by these operators. To see the significance of this result, we notice first that an arbitrty matrix $T\in\mathcal{M}_n$ does not satisfy the arithmetic-geometric mean inequality
\[{{s}_{j}}\left( \mathcal RT \right)\le \frac{1}{2}{{s}_{j}}\left( \left| T \right|\oplus \left| {{T}^{^{*}}} \right| \right).\] 
However, semi-hyponormal matrices satisfy the following.
	\begin{corollary}\label{cor_sing}
	Let $T\in\mathcal{M}_n$ be semi-hyponormal. Then 
	
	\[{{s}_{j}}\left( \mathcal RT \right)\le {{s}_{j}}\left( \left| T \right|\oplus \left| T \right| \right)\]
for $j=1,2,\ldots ,n$.
	\end{corollary}
	
	\begin{proof}
	This follows from \eqref{30}, with Lemma \ref{0}.
	\end{proof}
We remark that Corollary \ref{cor_sing} was shown already in \cite{hirz1}.

Now we move to the study of the other interesting class, namely  $(\alpha,\beta)$-normal operators.
\begin{theorem}
Let $T\in \mathcal B\left( \mathcal H \right)$ and $0< \alpha \le 1\le \beta $. Then the following are equivalent:
\begin{itemize}
\item[(i)] $$\left[ \begin{matrix}
   \frac{1}{{{\alpha }^{2}}}{{\left| {{T}^{*}} \right|}^{2}} & {{\left| T \right|}^{2}}  \\
   {{\left| T \right|}^{2}} & \frac{1}{{{\alpha }^{2}}}{{\left| {{T}^{*}} \right|}^{2}}  \\
\end{matrix} \right]\ge O\text{ and }\left[ \begin{matrix}
   {{\beta }^{2}}{{\left| T \right|}^{2}} & {{\left| {{T}^{*}} \right|}^{2}}  \\
   {{\left| {{T}^{*}} \right|}^{2}} & {{\beta }^{2}}{{\left| T \right|}^{2}}  \\
\end{matrix} \right]\ge O.$$
\item[(ii)] 
\begin{equation}\label{021}
{{\alpha }^{2}}{{\left| T \right|}^{2}}\le {{\left| {{T}^{*}} \right|}^{2}}\le {{\beta }^{2}}{{\left| T \right|}^{2}}.
\end{equation}
\end{itemize}

\end{theorem}
\begin{proof}
(ii)$\Rightarrow$(i) We show that if $\alpha^2|T|^2\leq |T^*|^2\leq\beta|T|^2$, then  
$$\left[ \begin{matrix}
   \frac{1}{{{\alpha }^{2}}}{{\left| {{T}^{*}} \right|}^{2}} & {{\left| T \right|}^{2}}  \\
   {{\left| T \right|}^{2}} & \frac{1}{{{\alpha }^{2}}}{{\left| {{T}^{*}} \right|}^{2}}  \\
\end{matrix} \right]\ge O.$$ By Lemma \ref{20}, it suffices to show that 
\begin{equation}\label{ned_pos}
\left|\left<|T|^2x,y\right>\right|^2\leq \left<\frac{1}{\alpha^2}|T^*|^2x,x\right>\left<\frac{1}{\alpha^2}|T^*|^2y,y\right>
\end{equation}
for all $x,y\in\mathcal{H}.$ By the mixed Cauchy-Schwarz inequality \cite{7}, we have, for $x,y\in\mathcal{H}$,
\begin{align*}
\left|\left<|T|^2x,y\right>\right|^2&\leq \left<|T|^2x,x\right>\left<|T|^2y,y\right>\\
&\leq \left<\frac{1}{\alpha^2}|T^*|^2x,x\right>\left<\frac{1}{\alpha^2}|T^*|^2y,y\right>
\end{align*}
where we have used the assumption $|T|^2\leq\frac{1}{\alpha^2}|T^*|^2$ to obtain the last inequality. This proves \eqref{ned_pos}, which is equivalent to $\left[ \begin{matrix}
   \frac{1}{{{\alpha }^{2}}}{{\left| {{T}^{*}} \right|}^{2}} & {{\left| T \right|}^{2}}  \\
   {{\left| T \right|}^{2}} & \frac{1}{{{\alpha }^{2}}}{{\left| {{T}^{*}} \right|}^{2}}  \\
\end{matrix} \right]\ge O.$ Following the same idea, we can show the positivity of the other matrix in (i).

(i)$\Rightarrow$(ii) Suppose that $\left[ \begin{matrix}
   \frac{1}{{{\alpha }^{2}}}{{\left| {{T}^{*}} \right|}^{2}} & {{\left| T \right|}^{2}}  \\
   {{\left| T \right|}^{2}} & \frac{1}{{{\alpha }^{2}}}{{\left| {{T}^{*}} \right|}^{2}}  \\
\end{matrix} \right]\ge O$. By Lemma \ref{20}, we get \eqref{ned_pos}. Hence $\left\langle {{\left| T \right|}^{2}}x,x \right\rangle \le \left\langle \frac{1}{{{\alpha }^{2}}}{{\left| {{T}^{*}} \right|}^{2}}x,x \right\rangle $, for all $x\in \mathcal H$. This proves ${{\alpha }^{2}}{{\left| T \right|}^{2}}\le {{\left| {{T}^{*}} \right|}^{2}}$. Following the same idea we show ${{\left| {{T}^{*}} \right|}^{2}}\le {{\beta }^{2}}{{\left| T \right|}^{2}}$.
\end{proof}

In Theorem \ref{3}, we showed a possible relation between the real part and the semi-hyponormal operator. In the following, we give the analogous result for $(\alpha,\beta)$-normal operator.

\begin{theorem}\label{28}
Let $0<\alpha\leq 1\leq \beta$ and let $T$ be a $(\alpha,\beta)$-normal operator. Then 
\[\pm \mathcal RT\le \sqrt[4]{\frac{\beta }{\alpha }}\left| T \right|\sharp \left| {{T}^{*}} \right|.\]
\end{theorem}
\begin{proof}
Let $T$ be a $(\alpha,\beta)$-normal operator and let $x,y \in \mathcal H$ be any vectors. Recall that the function $f\left( x \right)=\sqrt{x}$ is operator monotone on $\left[ 0,\infty  \right);$ \cite[Theorem 1.5.9 ]{2}. Therefore, by the mixed Schwarz inequality and the property of $(\alpha,\beta)$-normal operator, we have
\begin{equation}\label{01}
{{\left| \left\langle Tx,y \right\rangle  \right|}^{2}}\le \left\langle \left| T \right|x,x \right\rangle \left\langle \left| {{T}^{*}} \right|y,y \right\rangle \le \frac{1}{\alpha }\left\langle \left| {{T}^{*}} \right|x,x \right\rangle \left\langle \left| {{T}^{*}} \right|y,y \right\rangle
\end{equation}
and
\begin{equation}\label{02}
{{\left| \left\langle Tx,y \right\rangle  \right|}^{2}}\le \left\langle \left| T \right|x,x \right\rangle \left\langle \left| {{T}^{*}} \right|y,y \right\rangle \le \beta \left\langle \left| T \right|x,x \right\rangle \left\langle \left| T \right|y,y \right\rangle.
\end{equation}
Applying Lemma \ref{20} gives 
\begin{equation}\label{17}
\left[ \begin{matrix}
   \frac{1}{\sqrt{\alpha }}\left| {{T}^{*}} \right| & {{T}^{*}}  \\
   T & \frac{1}{\sqrt{\alpha }}\left| {{T}^{*}} \right|  \\
\end{matrix} \right]\ge O\quad{\text{and}}\quad\left[ \begin{matrix}
   \sqrt{\beta }\left| T \right| & {{T}^{*}}  \\
   T & \sqrt{\beta }\left| T \right|  \\
\end{matrix} \right]\ge O.
\end{equation}
It follows from \eqref{17} and Lemma \ref{16} that if $T$ is a $(\alpha,\beta)$-normal operator, then
\begin{equation}\label{27}
\left[ \begin{matrix}
   \sqrt[4]{\frac{\beta }{\alpha }}\left( \left| {{T}^{*}} \right|\sharp \left| T \right| \right) & {{T}^{*}}  \\
   T & \sqrt[4]{\frac{\beta }{\alpha }}\left( \left| {{T}^{*}} \right|\sharp \left| T \right| \right)  \\
\end{matrix} \right]\ge O.
\end{equation}
Because of Lemma \ref{20}, the above inequality has the following equivalent form
\begin{equation*}
{{\left| \left\langle Tx,y \right\rangle  \right|}^{2}}\le \sqrt{\frac{\beta }{\alpha }}\left\langle \left| T \right|\sharp \left| {{T}^{*}} \right|x,x \right\rangle \left\langle \left| T \right|\sharp \left| {{T}^{*}} \right|y,y \right\rangle,
\end{equation*}
for any vectors $x,y \in \mathcal H$. In particular,
\[\left| \left\langle Tx,x \right\rangle  \right|\le \sqrt[4]{\frac{\beta }{\alpha }}\left\langle \left| T \right|\sharp \left| {{T}^{*}} \right|x,x \right\rangle,\;\forall x\in\mathcal{H}.\]
Since
\begin{equation}\label{04}
\left\langle \pm \mathcal RTx,x \right\rangle =\pm \mathcal R\left\langle Tx,x \right\rangle \le \left| \left\langle Tx,x \right\rangle  \right|
\end{equation}
for any vectors $x \in \mathcal H$, it follows that 
\[\pm \mathcal RT\le \sqrt[4]{\frac{\beta }{\alpha }}\left| T \right|\sharp \left| {{T}^{*}} \right|\]
as desired.
\end{proof}

\begin{corollary}
Let $T\in\mathcal{M}_n$ be $(\alpha,\beta)$-normal. Then for $j=1,2,\cdots,n$, 
\[{{s}_{j}}\left( T \right)\le \sqrt[4]{\frac{\beta }{\alpha }}{{s}_{j}}\left( \left(\left| T \right|\sharp \left| {{T}^{*}} \right|\right)\oplus  \left(\left| T \right|\sharp \left| {{T}^{*}} \right|\right)\right).\] In particular, if $\|\cdot\|_u$ is an arbitrary unitarily invariant norm on $\mathcal{M}_n$, then 
\[{{\left\| T \right\|}_{u}}\le \sqrt[4]{\frac{\beta }{\alpha }}{{\left\|\; \left| T \right|\sharp \left| {{T}^{*}} \right|\; \right\|}_{u}}.\]
\end{corollary}
\begin{proof}
This follows from  \eqref{27} and Lemma \ref{0}.
\end{proof}

The following is a simpler relation between $\mathcal{R}T, |T|$ and $|T^*|$ for $(\alpha,\beta)$-normal operators.
\begin{theorem}\label{thm_real_del}
Let $0<\alpha\leq 1\leq \beta$ and let $T$ be a $(\alpha,\beta)$-normal operator. Then 
\[\pm \mathcal RT\le \sqrt{\beta }\left| T \right|\;{\text{and}}\; \pm\mathcal{R}T\leq \frac{1}{\sqrt{\alpha }}\left| {{T}^{*}} \right|.\]

\end{theorem}
\begin{proof}
It follows from \eqref{01}, that
\[\left| \left\langle Tx,x \right\rangle  \right|\le \frac{1}{\sqrt{\alpha }}\left\langle \left| {{T}^{*}} \right|x,x \right\rangle\]
for any vector $x\in \mathcal H$. Combining it with \eqref{04}, gives
\[\pm \mathcal RT\le \frac{1}{\sqrt{\alpha }}\left| {{T}^{*}} \right|.\]
In the same way, from \eqref{02}, we get
\[\pm \mathcal RT\le \sqrt{\beta }\left| T \right|.\]
This completes the proof.
\end{proof}

While $|\mathcal{R}T|$, in general, is not comparable with $|T|$ nor $|T^*|$, the following is an explicit comparison for $(\alpha,\beta)$-normal operators. We point out here that Theorem \ref{thm_real_del} deals with $\mathcal{R}T$, while Theorem \ref{15} deals with $|\mathcal{R}T|.$
\begin{theorem}\label{15}
Let $0<\alpha\leq 1\leq\beta$ and let $T$ be a $(\alpha,\beta)$-normal operator. Then
\[\left| \mathcal RT \right|\le \sqrt{\frac{1+{{\alpha }^{2}}}{2{{\alpha }^{2}}}}\left| {{T}^{*}} \right|\text{ and }\left| \mathcal RT \right|\le \sqrt{\frac{1+{{\beta }^{2}}}{2}}\left| T \right|.\]
\end{theorem}
\begin{proof}
According to the assumption
\[\begin{aligned}
   {{\left( \mathcal RT \right)}^{2}}&=\frac{1}{4}\left( {{\left| T \right|}^{2}}+{{\left| {{T}^{*}} \right|}^{2}}+2\mathcal R{{T}^{2}} \right) \\ 
 & \le \frac{1}{4}\left( {{\left| T \right|}^{2}}+{{\beta }^{2}}{{\left| T \right|}^{2}}+2\mathcal R{{T}^{2}} \right) \\ 
 & =\frac{1}{4}\left( \left( 1+{{\beta }^{2}} \right){{\left| T \right|}^{2}}+2\mathcal R{{T}^{2}} \right).  
\end{aligned}\]
The inequality stated above shows
\[\begin{aligned}
   {{\left( \mathcal RT \right)}^{2}}&\le \left( \frac{1+{{\beta }^{2}}}{2} \right){{\left| T \right|}^{2}}-{{\left( \mathcal IT \right)}^{2}} \\ 
 & \le \left( \frac{1+{{\beta }^{2}}}{2} \right){{\left| T \right|}^{2}}  
\end{aligned}\]
thanks to
\[{{\left( \mathcal RT \right)}^{2}}-\mathcal R{{T}^{2}}={{\left( \mathcal IT \right)}^{2}}.\] Then the second inequality follows by the
operator monotonicity of $f(x)=\sqrt{x}.$\\
The other inequality can be obtained by the same way.
\end{proof}
The following is an example explaining Theorem \ref{15}.
\begin{example}
Take $T=\left[ \begin{matrix}
   1 & 0  \\
   1 & 1  \\
\end{matrix} \right]$. Then $T$ will be $\left( \alpha ,\beta  \right)$-normal operator with ${{\alpha }^{2}}=\frac{3-\sqrt{5}}{2}$ and ${{\beta }^{2}}=\frac{3+\sqrt{5}}{2}$. Then 
$$\left| \mathcal RT \right|=\left[ \begin{matrix}
   1 & 0.5  \\
   0.5 & 1  \\
\end{matrix} \right]\lneqq	 \sqrt{\frac{1+{{\beta }^{2}}}{2}}\left| T \right|\approx \left[ \begin{matrix}
   1.8043 & 0.6014  \\
   0.6014 & 1.2028  \\
\end{matrix} \right].$$
The left term in the above inequality is also strictly less than
\[\sqrt{\frac{1+{{\alpha }^{2}}}{2{{\alpha }^{2}}}}\left| {{T}^{*}} \right|\approx \left[ \begin{matrix}
   1.2028 & 0.6014  \\
   0.6014 & 1.8043  \\
\end{matrix} \right].\]
\end{example}

\subsection{Two-operator inequalities}
In this section, we present various results treating two Hilbert space operators, continuing with the same theme of the paper. First, we have the following positive matrix operator.
\begin{lemma}\label{6}
Let $S,T\in \mathcal B\left( \mathcal H \right)$. If $f,g$ are non-negative continuous functions on $\left[ 0,\infty  \right)$ satisfying $f\left( t \right)g\left( t \right)=t$, $\left( t\ge 0 \right)$, then  
\[\left[ \begin{matrix}
   {{f}^{2}}\left( \left| T \right| \right) & {{T}^{*}}{{S}^{*}}  \\
   ST & S{{g}^{2}}\left( \left| {{T}^{*}} \right| \right){{S}^{*}}  \\
\end{matrix} \right]\ge O.\]
\end{lemma}
\begin{proof}
Let $x,y\in\mathcal{H}$. It follows from \cite[Theorem 1]{3} that
\[\begin{aligned}
   \left| \left\langle STx,y \right\rangle  \right|&=\left| \left\langle Tx,{{S}^{*}}y \right\rangle  \right| \\ 
 & \le \left\| f\left( \left| T \right| \right)x \right\|\left\| g\left( \left| {{T}^{*}} \right| \right){{S}^{*}}y \right\| \\ 
 & =\sqrt{\left\langle {{f}^{2}}\left( \left| T \right| \right)x,x \right\rangle \left\langle S{{g}^{2}}\left( \left| {{T}^{*}} \right| \right){{S}^{*}}y,y \right\rangle }.
\end{aligned}\]
Now, by Lemma \ref{20}, we get the desired result.
\end{proof}

Now we are ready to present the following inequality about $\mathcal{R}(ST)$.

\begin{theorem}\label{23}
Let $S,T\in \mathcal B\left( \mathcal H \right)$. If $f,g$ are non-negative continuous functions on $\left[ 0,\infty  \right)$ satisfying $f\left( t \right)g\left( t \right)=t$, $\left( t\ge 0 \right)$, then
\begin{equation}\label{14}
\pm \mathcal R\left( ST \right)\le \frac{\left(S{{g}^{2}}\left( \left| {{T}^{*}} \right| \right){{S}^{*}}\right)\sharp {{\left| {{S}^{*}} \right|}^{2}}+f\left( \left| T \right| \right)\left| T \right|}{2}.
\end{equation}
\end{theorem}
\begin{proof}\label{11}
Let $x,y\in\mathcal{H}$. For $S,T\in \mathcal B\left( \mathcal H \right)$, we have by the Cauchy-Schwartz inequality,
\[\begin{aligned}
   \left| \left\langle STx,y \right\rangle  \right|&=\left| \left\langle Tx,{{S}^{*}}y \right\rangle  \right| \\ 
 & \le \left\| Tx \right\|\left\| {{S}^{*}}y \right\| \\ 
 & =\sqrt{\left\langle {{\left| T \right|}^{2}}x,x \right\rangle \left\langle {{\left| {{S}^{*}} \right|}^{2}}y,y \right\rangle }
\end{aligned}\]
which implies, by Lemma \ref{1} and Lemma \ref{20},
\begin{equation}\label{9}
\left[ \begin{matrix}
   {{\left| {{S}^{*}} \right|}^{2}} & ST  \\
   {{T}^{*}}{{S}^{*}} & {{\left| T \right|}^{2}}  \\
\end{matrix} \right]\ge O.
\end{equation}
On the other hand, by Lemma \ref{6} and Lemma \ref{1}, we have 
\begin{equation}\label{8}
\left[ \begin{matrix}
   S{{g}^{2}}\left( \left| {{T}^{*}} \right| \right){{S}^{*}} & ST  \\
   {{T}^{*}}{{S}^{*}} & {{f}^{2}}\left( \left| T \right| \right)  \\
\end{matrix} \right]\ge O.
\end{equation}
The inequalities \eqref{9} and \eqref{8}, with Lemma \ref{16} give
\begin{equation}\label{13}
\left[ \begin{matrix}
   \left(S{{g}^{2}}\left( \left| {{T}^{*}} \right| \right){{S}^{*}}\right)\sharp{{\left| {{S}^{*}} \right|}^{2}} & ST  \\
   {{T}^{*}}{{S}^{*}} & {{f}^{2}}\left( \left| T \right| \right)\sharp{{\left| T \right|}^{2}}  \\
\end{matrix} \right]\ge O.
\end{equation}
Since ${{f}^{2}}\left( \left| T \right| \right){{\left| T \right|}^{2}}={{\left| T \right|}^{2}}{{f}^{2}}\left( \left| T \right| \right)$ (see \cite[Theorem 1.13]{high}), we have $f^2(|T|)\sharp|T|^2=f(|T|)|T|$. Now, noting \eqref{13}, then applying \eqref{eq_re_neg}, we get the desired result.
\end{proof}

\begin{remark}\label{18}
\hfill
\begin{itemize}
\item[(i)] Taking $S=I$, the identity operator, in Theorem \ref{23}, then we have
\[\pm\mathcal R T \le \frac{g\left( \left| {{T}^{*}} \right| \right)+{{f}}\left( \left| T \right| \right){{\left| T \right|}}}{2}.\]
\item[(ii)] If we replace $S$ by $\textup i{{S}^{*}}$, in Theorem \ref{23}, we get
\[\pm\mathcal I\left( {{T}^{*}}S \right)\le \frac{\left({{S}^{*}}g^{2}\left( \left| {{T}^{*}} \right| \right)S\right)\sharp {{\left| S \right|}^{2}}+{{f}}\left( \left| T \right| \right) {{\left| T \right|}}}{2}.\]
\end{itemize}
\end{remark}

The next result follows from Remark \ref{18} (ii) and Theorem \ref{23}.
\begin{corollary}
Let $S,T\in \mathcal B\left( \mathcal H \right)$. If $f,g$ are non-negative continuous functions on $\left[ 0,\infty  \right)$ satisfying $f\left( t \right)g\left( t \right)=t$, $\left( t\ge 0 \right)$, then
\[\pm \left( \mathcal R\left( S{{T}^{*}} \right)+\mathcal I\left( T{{S}^{*}} \right) \right)\le \left(Sg^{2}\left( \left| T \right| \right){{S}^{*}}\right)\sharp {{\left| {{S}^{*}} \right|}^{2}}+{{f}}\left( \left| {{T}^{*}} \right| \right) {{\left| {{T}^{*}} \right|}}.\]
In particular,
\[\pm \left( \mathcal R T+\mathcal I T  \right)\le g\left( \left| T \right| \right)+{{f}}\left( \left| {{T}^{*}} \right| \right) {{\left| {{T}^{*}} \right|}}.\]
\end{corollary}

Theorem \ref{23} can be extended to the sum of operators by using the linearity of $\mathcal R$. 
\begin{corollary}
Let ${{T}_{1}},{{T}_{2}},{{T}_{3}},{{T}_{4}}\in \mathcal B\left( \mathcal H \right)$. If $f,g$ are non-negative continuous functions on $\left[ 0,\infty  \right)$ satisfying $f\left( t \right)g\left( t \right)=t$, $\left( t\ge 0 \right)$, then 
\begin{equation}\label{2}
\begin{aligned}
  & \pm \mathcal R\left( {{T}_{1}}{{T}_{2}}\pm {{T}_{3}}{{T}_{4}} \right) \\ 
 & \le \frac{1}{2}\left( \left({{T}_{1}}{{g}^{2}}\left( \left| {{{{T}_{2}}}^{*}} \right| \right){{{{T}_{1}}}^{*}}\right)\sharp{{\left| {{{{T}_{1}}}^{*}} \right|}^{2}}+\left({{T}_{3}}{{g}^{2}}\left( \left| {{{{T}_{4}}}^{*}} \right| \right){{{{T}_{3}}}^{*}}\right)\sharp{{\left| {{{{T}_{3}}}^{*}} \right|}^{2}}+{{f}}\left( \left| {{T}_{2}} \right| \right){{\left| {{T}_{2}} \right|}}+{{f}}\left( \left| {{T}_{4}} \right| \right){{\left| {{T}_{4}} \right|}} \right).  
\end{aligned}
\end{equation}
In particular,
\begin{equation}\label{5}
\pm \mathcal R\left( {{T}_{1}}{{T}_{2}}\pm {{T}_{3}}{{T}_{4}} \right)\le \frac{1}{2}\left( \left({{T}_{1}}\left| {{{{T}_{2}}}^{*}} \right|{{{{T}_{1}}}^{*}}\right)\sharp{{\left| {{{{T}_{1}}}^{*}} \right|}^{2}}+\left({{T}_{3}}\left| {{{{T}_{4}}}^{*}} \right|{{{{T}_{3}}}^{*}}\right)\sharp{{\left| {{{{T}_{3}}}^{*}} \right|}^{2}}+{{\left| {{T}_{2}} \right|}^{\frac{3}{2}}}+{{\left| {{T}_{4}} \right|}^{\frac{3}{2}}} \right).
\end{equation}
\end{corollary}
\begin{proof}
From Theorem \ref{23}, we have
$$\begin{aligned}
  & \pm \mathcal R\left( {{T}_{1}}{{T}_{2}}+{{T}_{3}}{{T}_{4}} \right) \\ 
 & =\pm \mathcal R\left( {{T}_{1}}{{T}_{2}} \right)\pm \mathcal R\left( {{T}_{3}}{{T}_{4}} \right) \\ 
 & \le \frac{1}{2}\left(\left( {{T}_{1}}{{g}^{2}}\left( \left| {{{{T}_{2}}}^{*}} \right| \right){{{{T}_{1}}}^{*}}\right)\sharp{{\left| {{{{T}_{1}}}^{*}} \right|}^{2}}+\left({{T}_{3}}{{g}^{2}}\left( \left| {{{{T}_{4}}}^{*}} \right| \right){{{{T}_{3}}}^{*}}\right)\sharp{{\left| {{{{T}_{3}}}^{*}} \right|}^{2}}+{{f}}\left( \left| {{T}_{2}} \right| \right){{\left| {{T}_{2}} \right|}}+{{f}}\left( \left| {{T}_{4}} \right| \right){{\left| {{T}_{4}} \right|}} \right).  
\end{aligned}$$
This proves \eqref{2}.
The inequality \eqref{5} follows from \eqref{2} by setting $f\left( x \right)=g\left( x \right)=\sqrt{x}$.
\end{proof}

\begin{corollary}\label{19}
Let $S,T\in \mathcal{M}_n$. Then, for any $0\le t',v'\le \frac{3}{2}$,
\[\pm \mathcal R\left( S\pm T \right)\le \frac{1}{2}\left( {{\left| S \right|}^{t'}}+{{\left| {{S}^{*}} \right|}^{2-t'}}+{{\left| T \right|}^{v'}}+{{\left| {{T}^{*}} \right|}^{2-v'}} \right).\]
\end{corollary}
\begin{proof}
Let $S=U\left| S \right|$ and $T=V\left| T \right|$ be the polar decomposition of $S$ and $T$, respectively. For $0\leq t',v'\leq\frac{3}{2},$ let $t=\frac{2}{3}t,$ and $v=\frac{2}{3}v'.$ Then $0\leq t,v\leq 1.$ If we put $T_1=U{{\left| S \right|}^{1-t}}$, $T_2={{\left| S \right|}^{t}}$, $T_3=V{{\left| T \right|}^{1-v}}$, and $T_4={{\left| T \right|}^{v}}$, where $0\le t,v\le 1$, in the inequality \eqref{5}, we get
\[\pm \mathcal R\left( S\pm T \right)\le \frac{1}{2}\left( {{\left| S \right|}^{t'}}+{{\left| {{S}^{*}} \right|}^{2-t'}}+{{\left| T \right|}^{v'}}+{{\left| {{T}^{*}} \right|}^{2-v'}} \right),\text{ }0\le t',v'\le \frac{3}{2}\]
as desired.
\end{proof}

\begin{remark}
Taking $T=\textup iT$, in Corollary \ref{19}, we infer that
\[\pm \mathcal R\left( S\pm \textup iT \right)\le \frac{1}{2}\left( {{\left| S \right|}^{t'}}+{{\left| {{S}^{*}} \right|}^{2-t'}}+{{\left| T \right|}^{v'}}+{{\left| {{T}^{*}} \right|}^{2-v'}} \right).\]
In particular,
\[\mathcal \pm \mathcal RT\le \frac{1}{2}\left( {{\left| \mathcal RT \right|}^{t'}}+{{\left| \mathcal RT \right|}^{2-t'}}+{{\left| \mathcal IT \right|}^{v'}}+{{\left| \mathcal IT \right|}^{2-v'}} \right)\]
due to the Cartesian decomposition of $T$  (for any $X\in \mathcal B\left( \mathcal H \right)$, we have $X=\mathcal RX+\textup i\mathcal IX$).
\end{remark}

\begin{remark}
Lemma \ref{tao}, together with  \eqref{13}, implies
\[2{{s}_{j}}\left( ST \right)\le {{s}_{j}}\left( \left[ \begin{matrix}
   \left(S{{g}^{2}}\left( \left| {{T}^{*}} \right| \right){{S}^{*}}\right)\sharp {{\left| {{S}^{*}} \right|}^{2}} & ST  \\
   {{T}^{*}}{{S}^{*}} & f\left( \left| T \right| \right)\left| T \right|  \\
\end{matrix} \right] \right),\text{ }j=1,2,\ldots ,n,\]
for $S,T\in\mathcal{M}_n$.
This implies for any unitarily invariant norm ${{\left\| \cdot \right\|}_{u}}$ on $\mathcal{M}_n$,
\[{{\left\| ST \right\|}_{u}}\le \frac{1}{2}{{\left\| \left[ \begin{matrix}
  \left( S{{g}^{2}}\left( \left| {{T}^{*}} \right| \right){{S}^{*}}\right)\sharp {{\left| {{S}^{*}} \right|}^{2}} & ST  \\
   {{T}^{*}}{{S}^{*}} & f\left( \left| T \right| \right)\left| T \right|  \\
\end{matrix} \right] \right\|}_{u}}.\]

It is shown in \cite{14} that Lemma \ref{0} is equivalent to the following: If $A$ is self-adjoint, $B$ is positive, and $\pm A\le B$, then 
	\[{{s}_{j}}\left( A \right)\le {{s}_{j}}\left( B\oplus B \right)\]
for $j=1,2,\ldots ,n$. Thus, for $S,T\in\mathcal{M}_n$, Corollary \ref{19} gives
\[\begin{aligned}
  & {{s}_{j}}\left( \mathcal R\left( S\pm T \right) \right) \\ 
 & \le \frac{1}{2}{{s}_{j}}\left( {{\left| S \right|}^{t'}}+{{\left| {{S}^{*}} \right|}^{2-t'}}+{{\left| T \right|}^{v'}}+{{\left| {{T}^{*}} \right|}^{2-v'}}\oplus {{\left| S \right|}^{t'}}+{{\left| {{S}^{*}} \right|}^{2-t'}}+{{\left| T \right|}^{v'}}+{{\left| {{T}^{*}} \right|}^{2-v'}} \right). \\ 
\end{aligned}\]
Letting $t'=v'=1$ and $T=O$, we get
\[{{s}_{j}}\left( \mathcal RS \right)\le \frac{1}{2}{{s}_{j}}\left( {{\left| S \right|}}+{{\left| {{S}^{*}} \right|}}\oplus {{\left| S \right|}}+{{\left| {{S}^{*}} \right|}} \right)\]
which is a known result \cite[Theorem 6]{15}.
\end{remark}

\begin{corollary}\label{cor_sing_ST}
Let $T_1,T_2,T_3,T_4\in\mathcal{M}_n$ and let $f,g:[0,\infty)\to [0,\infty)$ be non-negative continuous functions such that $f(t)g(t)=t.$ Then for $j=1,\cdots,n,$

\begin{equation}\label{o1}
\begin{aligned}
  & {{s}_{j}}\left( {{T}_{1}}{{T}_{2}}\pm {{T}_{3}}{{T}_{4}} \right) \\ 
 & \le {{s}_{j}}\left( {{T}_{1}}{{g}^{2}}\left( \left| {{{{T}_{2}}}^{*}} \right| \right){{{{T}_{1}}}^{*}}\sharp  {{\left| {{{{T}_{1}}}^{*}} \right|}^{2}}+{{T}_{3}}{{g}^{2}}\left( \left| {{{{T}_{4}}}^{*}} \right| \right){{{{T}_{3}}}^{*}}\sharp  {{\left| {{{{T}_{3}}}^{*}} \right|}^{2}}\oplus f\left( \left| {{T}_{2}} \right| \right)\left| {{T}_{2}} \right|+f\left( \left| {{T}_{4}} \right| \right)\left| {{T}_{4}} \right| \right). \\ 
\end{aligned}
\end{equation}

In particular, if $S,T\in\mathcal{M}_n$, then for $j=1,\cdots,n,$
\[{{s}_{j}}\left( {S}\pm \textup i{T} \right)\le {{s}_{j}}\left( g\left( \left| {{{S}}^{*}} \right| \right)+g\left( \left| {{{T}}^{*}} \right| \right)\oplus f\left( \left| {S} \right| \right)\left| {S} \right|+f\left( \left| {T} \right| \right)\left| {T} \right| \right)\]
and
\[{{s}_{j}}\left( {S}\pm \textup i{T} \right)\le {{s}_{j}}\left( \left| {{{S}}^{*}} \right|+\left| {{{T}}^{*}} \right|\oplus \left| {S} \right|+\left| {T} \right| \right).\]	
\end{corollary}
\begin{proof}
If $T_1,T_2,T_3,T_4\in\mathcal{B}(\mathcal{H})$, then by  \eqref{13}, we have
	\[\left[ \begin{matrix}
   {{T}_{1}}{{g}^{2}}\left( \left| {{{{T}_{2}}}^{*}} \right| \right){{{{T}_{1}}}^{*}}\sharp  {{\left| {{{{T}_{1}}}^{*}} \right|}^{2}} & {{T}_{1}}{{T}_{2}}  \\
   {{{{T}_{2}}}^{*}}{{{{T}_{1}}}^{*}} & f\left( \left| {{T}_{2}} \right| \right)\left| {{T}_{2}} \right|  \\
\end{matrix} \right]\ge O,\]
and
	\[\left[ \begin{matrix}
   {{T}_{3}}{{g}^{2}}\left( \left| {{{{T}_{4}}}^{*}} \right| \right){{{{T}_{3}}}^{*}}\sharp  {{\left| {{{{T}_{3}}}^{*}} \right|}^{2}} & \pm {{T}_{3}}{{T}_{4}}  \\
   \pm {{{{T}_{4}}}^{*}}{{{{T}_{3}}}^{*}} & f\left( \left| {{T}_{4}} \right| \right)\left| {{T}_{4}} \right|  \\
\end{matrix} \right]\ge O.\]
Therefore,
\[\left[ \begin{matrix}
   {{T}_{1}}{{g}^{2}}\left( \left| {{{{T}_{2}}}^{*}} \right| \right){{{{T}_{1}}}^{*}}\sharp  {{\left| {{{{T}_{1}}}^{*}} \right|}^{2}}+{{T}_{3}}{{g}^{2}}\left( \left| {{{{T}_{4}}}^{*}} \right| \right){{{{T}_{3}}}^{*}}\sharp  {{\left| {{{{T}_{3}}}^{*}} \right|}^{2}} & {{T}_{1}}{{T}_{2}}\pm {{T}_{3}}{{T}_{4}}  \\
   {{{{T}_{2}}}^{*}}{{{{T}_{1}}}^{*}}\pm {{{{T}_{4}}}^{*}}{{{{T}_{3}}}^{*}} & f\left( \left| {{T}_{2}} \right| \right)\left| {{T}_{2}} \right|+f\left( \left| {{T}_{4}} \right| \right)\left| {{T}_{4}} \right|  \\
\end{matrix} \right]\ge O.\]
By Lemma \ref{0}, we get
\[
\begin{aligned}
  & {{s}_{j}}\left( {{T}_{1}}{{T}_{2}}\pm {{T}_{3}}{{T}_{4}} \right) \\ 
 & \le {{s}_{j}}\left( {{T}_{1}}{{g}^{2}}\left( \left| {{{{T}_{2}}}^{*}} \right| \right){{{{T}_{1}}}^{*}}\sharp  {{\left| {{{{T}_{1}}}^{*}} \right|}^{2}}+{{T}_{3}}{{g}^{2}}\left( \left| {{{{T}_{4}}}^{*}} \right| \right){{{{T}_{3}}}^{*}}\sharp  {{\left| {{{{T}_{3}}}^{*}} \right|}^{2}}\oplus f\left( \left| {{T}_{2}} \right| \right)\left| {{T}_{2}} \right|+f\left( \left| {{T}_{4}} \right| \right)\left| {{T}_{4}} \right| \right) \\ 
\end{aligned}
\]
which proves \eqref{o1}. \\
For the other two inequalities, let ${{T}_{1}}={{T}_{3}}=I$, ${{T}_{4}}=\textup iT$, and ${{T}_{2}}=S$, in \eqref{o1}. Then
	\[{{s}_{j}}\left( {S}\pm \textup i{T} \right)\le {{s}_{j}}\left( g\left( \left| {{{S}}^{*}} \right| \right)+g\left( \left| {{{T}}^{*}} \right| \right)\oplus f\left( \left| {S} \right| \right)\left| {S} \right|+f\left( \left| {T} \right| \right)\left| {T} \right| \right).\]
By letting $f(t)=t^{\mu}$ and $g(t)=t^{1-\mu}$ for $0\leq\mu\leq1$ in the latter inequality we obtain 
\[{{s}_{j}}\left( {S}\pm \textup i{T} \right)\le {{s}_{j}}\left( {{\left| {{{S}}^{*}} \right|}^{1-\mu}}+{{\left| {{{T}}^{*}} \right|}^{1-\mu}}\oplus {{\left| {S} \right|}^{1+\mu}}+{{\left| {T} \right|}^{1+\mu}} \right),\]
which implies the last required inequality by letting $\mu=0.$
This completes the proof.
\end{proof}

\begin{remark}
A related inequality to Corollary \ref{cor_sing_ST}, in \cite{10} it has been shown that for positive semi-definite matrices $S$ and $T$, one has	
\[{{s}_{j}}\left( S+\textup iT \right)\le {{s}_{j}}\left( S+T \right).\]
\end{remark}	 

\section{Numerical Radius Inequalities}
Theorem \ref{23} can be utilized to obtain an upper bound for the numerical radius of the product of  two operators. An exciting application of the following result can be seen in Remark \ref{rem_w_product} below. We remark that in recent years, a considerable research work is devoted to numerical radius inequalities, as one can see in \cite{ bks,bhunia,feki,r2,zamani}.
\begin{corollary}\label{10}
Let $S,T\in \mathcal B\left( \mathcal H \right)$. If $f,g$ are non-negative continuous functions on $\left[ 0,\infty  \right)$ satisfying $f\left( t \right)g\left( t \right)=t$, $\left( t\ge 0 \right)$, then
\[\omega \left( ST \right)\le \frac{1}{2}\left\| S{{g}^{2}}\left( \left| {{T}^{*}} \right| \right){{S}^{*}}\sharp{{\left| {{S}^{*}} \right|}^{2}}+{{f}}\left( \left| T \right| \right){{\left| T \right|}} \right\|.\]
In particular,
\begin{equation}\label{12}
\omega \left( ST \right)\le \frac{1}{2}\left\| S{{\left| {{T}^{*}} \right|}^{2\left( 1-t \right)}}{{S}^{*}}\sharp{{\left| {{S}^{*}} \right|}^{2}}+{{\left| T \right|}^{1+t}} \right\|,0\le t\le 1.
\end{equation}
\end{corollary}
\begin{proof}
If we replace $S$ by ${{e}^{\textup i\theta }}S$, $\left( \theta \in \mathbb{R} \right)$, in \eqref{14}, we get
\[\mathcal R\left( {{e}^{\textup i\theta }}ST \right)\le \frac{S{{g}^{2}}\left( \left| {{T}^{*}} \right| \right){{S}^{*}}\sharp {{\left| {{S}^{*}} \right|}^{2}}+{{f}}\left( \left| T \right| \right) {{\left| T \right|}}}{2}.\]
This operator inequality  implies to the following norm inequality
	\[\left\| \mathcal R\left( {{e}^{\textup i\theta }}ST \right) \right\|\le \frac{1}{2}\left\| S{{g}^{2}}\left( \left| {{T}^{*}} \right| \right){{S}^{*}}\sharp {{\left| {{S}^{*}} \right|}^{2}}+{{f}}\left( \left| T \right| \right) {{\left| T \right|}} \right\|.\]
Now, the result follows by taking supremum over $\theta \in \mathbb{R}$, since \cite{yamazaki}
	\[\underset{\theta \in \mathbb{R}}{\mathop{\sup }}\,\left\| \mathcal R{{e}^{\textup i\theta }}A \right\|=\omega \left( A \right).\]
The inequality \eqref{12} follows from the above inequality by taking $g\left( x \right)={{x}^{1-t}}$ and $f\left( x \right)={{x}^{t}}$ with $ 0\le t\le 1$.
\end{proof}
\begin{remark}\label{rem_w_product}
 In this remark we explain the importance of Corollary \ref{10}, where we retrieve two celebrated inequalities for the numerical radius. Thus, Corollary \ref{10} is a generalized form that can be used to obtain several inequalities upon choosing appropriate functions and parameters.
\hfill
\begin{itemize}
\item[(i)] The case $t=1$ in the inequality \eqref{12}, reduces to
\[\omega \left( ST \right)\le \frac{1}{2}\left\| {{\left| {{S}^{*}} \right|}^{2}}+{{\left| T \right|}^{2}} \right\|.\]
The above inequality has been given in \cite{4}.
\item[(ii)] If we set $S=I$ and $t=0$ in the inequality \eqref{12}, we get
\[\omega \left( T \right)\le \frac{1}{2}\left\| \left| T \right|+\left| {{T}^{*}} \right| \right\|\]
which was proved in \cite{5}.
\end{itemize}
\end{remark}

\begin{remark}
It must be emphasized that inequality \eqref{12} provides a non-trivial estimates for the numerical radius of the product of two operators. To show this, we recall that Kittaneh \cite[Theorem 2]{4}  proved if ${{T}_{i}}\in \mathcal B\left( \mathcal H \right)$, $\left( i=1,2,\ldots ,6 \right)$, then for any $0\le t\le 1$,
{\small
\begin{equation}\label{26}
\omega \left( {{T}_{1}}{{T}_{2}}{{T}_{3}}+{{T}_{4}}{{T}_{5}}{{T}_{6}} \right)\le \frac{1}{2}\left\| {{T}_{1}}{{\left| T_{2}^{*} \right|}^{2\left( 1-t \right)}}T_{1}^{*}+T_{3}^{*}{{\left| T_{2} \right|}^{2t}}{{T}_{3}}+{{T}_{4}}{{\left| T_{5}^{*} \right|}^{2\left( 1-t \right)}}T_{4}^{*}+T_{6}^{*}{{\left| {{T}_{5}} \right|}^{2t}}{{T}_{6}} \right\|.
\end{equation}}
Setting ${{T}_{1}}=S$, ${{T}_{2}}=T$, ${{T}_{3}}=I$, and ${{T}_{5}}=O$, in \eqref{26}. Therefore, we have
\begin{equation}\label{24}
\omega \left( ST \right)\le \frac{1}{2}\left\| S{{\left| {{T}^{*}} \right|}^{2\left( 1-t \right)}}{{S}^{*}}+{{\left| T \right|}^{2t}} \right\|.
\end{equation}
The case ${{T}_{1}}=S$, ${{T}_{2}}=I$, ${{T}_{3}}=T$, ${{T}_{5}}=O$, also implies
\begin{equation}\label{25}
\omega \left( ST \right)\le \frac{1}{2}\left\| {{\left| {{S}^{*}} \right|}^{2}}+{{\left| T \right|}^{2}} \right\|.
\end{equation}
From the relations \eqref{24} and \eqref{25}, we infer that
\begin{equation}\label{31}
\omega \left( ST \right)\le \frac{1}{4}\left\| S{{\left| {{T}^{*}} \right|}^{2\left( 1-t \right)}}{{S}^{*}}+{{\left| T \right|}^{2t}} \right\|+\frac{1}{4}\left\| {{\left| {{S}^{*}} \right|}^{2}}+{{\left| T \right|}^{2}} \right\|.
\end{equation}
On the other hand, it follows from the inequality \eqref{12} that
\[\begin{aligned}
   \omega \left( ST \right)&\le \frac{1}{2}\left\| \left( S{{\left| {{T}^{*}} \right|}^{2\left( 1-t \right)}}{{S}^{*}}\sharp {{\left| {{S}^{*}} \right|}^{2}} \right)+{{\left| T \right|}^{1+t}} \right\| \\ 
 & =\frac{1}{2}\left\| \left( S{{\left| {{T}^{*}} \right|}^{2\left( 1-t \right)}}{{S}^{*}}\sharp {{\left| {{S}^{*}} \right|}^{2}} \right)+\left( {{\left| T \right|}^{2t}}\sharp {{\left| T \right|}^{2}} \right) \right\| \\ 
 & \le \frac{1}{2}\left\| \left( S{{\left| {{T}^{*}} \right|}^{2\left( 1-t \right)}}{{S}^{*}}+{{\left| T \right|}^{2t}} \right)\sharp \left( {{\left| {{S}^{*}} \right|}^{2}}+{{\left| T \right|}^{2}} \right) \right\| \\
 &\quad \text{(by \cite[Corollary I. 2.1]{10})}\\ 
 & \le \frac{1}{4}\left\| S{{\left| {{T}^{*}} \right|}^{2\left( 1-t \right)}}{{S}^{*}}+{{\left| T \right|}^{2t}}+{{\left| {{S}^{*}} \right|}^{2}}+{{\left| T \right|}^{2}} \right\| \\ 
  &\quad \text{(by the operator arithmetic-geometric mean inequality)}\\
 & \le \frac{1}{4}\left\| S{{\left| {{T}^{*}} \right|}^{2\left( 1-t \right)}}{{S}^{*}}+{{\left| T \right|}^{2t}} \right\|+\frac{1}{4}\left\| {{\left| {{S}^{*}} \right|}^{2}}+{{\left| T \right|}^{2}} \right\|.  
\end{aligned}\]
Indeed our estimate \eqref{12} is a refinement of the inequality \eqref{31}, obtained via Kittaneh inequality \eqref{26}.
\end{remark}

\begin{theorem}
Let $T$ be $\left( \alpha ,\beta  \right)$-normal operator. Then
\begin{equation}\label{22}
\max \left\{ \sqrt{1+\frac{1}{{{\beta }^{2}}}},\sqrt{1+{{\alpha }^{2}}} \right\}\frac{1}{2}\left\| T \right\|\le \omega \left( T \right).
\end{equation}
\end{theorem}
\begin{proof}
We have
\[\begin{aligned}
   \frac{1+{{\alpha }^{2}}}{2}{{\left\| T \right\|}^{2}}&=\frac{2\left( 1+{{\alpha }^{2}} \right)}{4}\left\| \;{{\left| T \right|}^{2}} \;\right\| \\ 
 & \le \frac{1}{4}\left( \left\| \left( 1+{{\alpha }^{2}} \right){{\left| T \right|}^{2}}+{{T}^{2}}+{{\left( {{T}^{*}} \right)}^{2}} \right\|+\left\| \left( 1+{{\alpha }^{2}} \right){{\left| T \right|}^{2}}-\left( {{T}^{2}}+{{\left( {{T}^{*}} \right)}^{2}} \right) \right\| \right) \\ 
 & \le \frac{1}{4}\left( \left\| {{\left| T \right|}^{2}}+{{\left| {{T}^{*}} \right|}^{2}}+{{T}^{2}}+{{\left( {{T}^{*}} \right)}^{2}} \right\|+\left\| {{\left| T \right|}^{2}}+{{\left| {{T}^{*}} \right|}^{2}}-\left( {{T}^{2}}+{{\left( {{T}^{*}} \right)}^{2}} \right) \right\| \right) \\ 
 & =\left\| {{\left( \frac{T+{{T}^{*}}}{2} \right)}^{2}} \right\|+\left\| {{\left( \frac{T-{{T}^{*}}}{2\textup i} \right)}^{2}} \right\| \\ 
 & ={{\left\| \mathcal RT \right\|}^{2}}+{{\left\| \mathcal IT \right\|}^{2}} \\ 
 & \le 2{{\omega }^{2}}\left( T \right),
\end{aligned}\]
where the triangle inequality for the usual operator norm implies the first inequality, the second inequality is achieved because of the first inequality in \eqref{21}, and the last inequality is obtained from the fact that
\[\left| \left\langle Tx,x \right\rangle  \right|=\sqrt{{{\left\langle \mathcal RTx,x \right\rangle }^{2}}+{{\left\langle \mathcal ITx,x \right\rangle }^{2}}};\text{ }x\in \mathcal H,\left\| x \right\|=1.\]
Consequently,
\[\sqrt{1+{{\alpha }^{2}}}\;\frac{1}{2}\left\| T \right\|\le \omega \left( T \right).\]
In the same way, one can show that
\[\sqrt{1+\frac{1}{{{\beta }^{2}}}}\;\frac{1}{2}\left\| T \right\|\le \omega \left( T \right).\]
Combining the above two inequalities implies the desired result \eqref{22}.
\end{proof}

Of course, $1\le \max \left\{ \sqrt{1+\frac{1}{{{\beta }^{2}}}},\sqrt{1+{{\alpha }^{2}}} \right\}$. Consequently, \eqref{22} improves considerably the following well-known inequality for general operators
\begin{equation}\label{eq_1/2norm}
	\frac{1}{2}\left\| T \right\|\le \omega \left( T \right).
\end{equation}

It is also interesting to note that in the case of the hyponormal operator (in the sense that ${{\left| T \right|}^{2}}\le {{\left| {{T}^{*}} \right|}^{2}}$), by applying the same method as in the above (in fact $\alpha =1$), we get the following refinement of \eqref{eq_1/2norm} for hyponormal operators
\[\frac{1}{\sqrt{2}}\left\| T \right\|\le \omega \left( T \right).\]

We remark that the last inequality is known for accretive-dissipative operators, see \cite{9}.

\begin{remark}
We remark here that \cite[(3.1)]{8} should be written in the following correct form 
\[\left( 1+{{\alpha }^{2}} \right){{\left\| T \right\|}^{2}}\le \frac{1}{2}\left( {{\left\| T+{{T}^{*}} \right\|}^{2}}+{{\left\| T-{{T}^{*}} \right\|}^{2}} \right).\]
\end{remark}

\section*{Acknowledgment} The authors would like to express their sincere thanks to the anonymous reviewers, whose comments have considerably improved the quality of this paper. Also, the authors are grateful to Professor J. Bourin for some private communications related to this work.

\vskip 0.3 true cm 

\noindent{\tiny (M. Sababheh) Department of basic sciences, Princess Sumaya University for Technology, Amman, Jordan}
	
\noindent	{\tiny\textit{E-mail address:} sababheh@psut.edu.jo; sababheh@yahoo.com}

\vskip 0.3 true cm

\noindent{\tiny (H. R. Moradi) Department of Mathematics, Payame Noor University (PNU), P.O. Box, 19395-4697, Tehran, Iran
	
\noindent	\textit{E-mail address:} hrmoradi@mshdiau.ac.ir}

%-----------------------------------------------------------------------------
%-----------------------------------------------------------------------------
\end{document}